\documentclass[reqno]{amsart}
\usepackage{amsmath,amsthm,amssymb,latexsym,fullpage,setspace,graphicx,float,xcolor,hyperref,verbatim,mathtools}
\usepackage{tikz-cd,tikz,pgfplots}
\usepackage[mode=buildnew]{standalone}
\usetikzlibrary{decorations.markings,math,positioning,graphs}
\usepackage[utf8]{inputenc}
\usepackage[english]{babel}

\definecolor{myBlue}{RGB}{66,103,135}
\definecolor{myLightBlue}{RGB}{169,205,244}
\definecolor{myRed}{RGB}{242,95,93}
\definecolor{myOrange}{RGB}{249,154,90}
\definecolor{myTan}{RGB}{254,234,173}

\theoremstyle{plain}
\newtheorem{thm}{Theorem}[section]
\newtheorem{cor}[thm]{Corollary}

\newtheorem{prop}[thm]{Proposition}

\theoremstyle{definition}
\newtheorem{defn}[thm]{Definition}
\newtheorem{eg}[thm]{Example}

\theoremstyle{remark}
\newtheorem{remark}[thm]{Remark}

\makeatletter
\@namedef{subjclassname@2020}{%
  \textup{2020} Mathematics Subject Classification}
\makeatother

\title{Superoptimal continued fractions}
\author{Slade Sanderson}
\address{Universit\'e Paris Cit\'e, CNRS, IRIF, F-75013, Paris, France}
\email{slade.sanderson@irif.fr}
\date{\today}

\def\G{
\mathcal{G}
}

\begin{document}

\begin{abstract}
Motivated by the optimal continued fractions studied independently by Selenius and Bosma, we define and introduce algorithms producing superoptimal continued fraction expansions of irrationals.  The convergents of these expansions simultaneously provide arbitrarily good rational approximations and converge arbitrarily quickly.
\end{abstract}

\maketitle
\tikzset{->-/.style={decoration={markings,mark=at position #1 with {\arrow{>}}},postaction={decorate}}}

\section{Introduction}\label{Introduction}

Every irrational number $x$ has a unique \emph{regular continued fraction} ({\sc rcf}) expansion
\begin{equation}\label{rcf_expn}
x=[a_0;a_1,a_2,\dots]=a_0+\cfrac{1}{a_1+\cfrac{1}{a_2+\ddots}},
\end{equation}
where each \emph{partial denominator} $a_n$ is an integer, and $a_n$ is positive for $n>0$.  This expansion is interpreted as the limit of the \emph{{\sc rcf}-convergents} of $x$, which are the (reduced) rationals
\begin{equation}\label{convergents}
\frac{p_n}{q_n}=[a_0;a_1,a_2,\dots,a_n]=a_0+\cfrac{1}{a_1+\cfrac{1}{\ddots+\cfrac{1}{a_n}}}
\end{equation}
obtained by truncating the expression in \eqref{rcf_expn} after finitely many partial denominators.  Conversely, for any sequence of integers $(a_n)_{n\ge 0}$ with $a_n$ positive for $n>0$, the rationals defined by \eqref{convergents} converge to a unique irrational $x$; see, e.g., \cite{K97}.  The partial denominators of {\sc rcf}-expansions are obtained algorithmically via the \emph{Gauss map} $G:[0,1)\to[0,1)$ defined by $G(0)\coloneqq 0$ and $G(x)\coloneqq 1/x-\lfloor 1/x\rfloor$ for $x\neq 0$.  Indeed, one finds that $a_0=\lfloor x\rfloor$ and $a_n=\lfloor 1/G^{n-1}(x_0)\rfloor$ for all $n>0$, where $x_0\coloneqq x-a_0$.  Moreover, $G$ acts as a one-sided shift on {\sc rcf}-expansions: if $x=[0;a_1,a_2,\dots]$, then $G(x)=[0;a_2,a_3,\dots]$.  The \emph{Gauss measure} $\nu_G$ is the ergodic (in fact, exact), absolutely continuous, $G$-invariant probability measure with density $1/(\log(2)(1+x))$.  Together with the pointwise ergodic theorem, the Gauss measure reveals many asymptotic properties of generic {\sc rcf}-expansions; see \cite{DK2002B}.

Despite its immense usefulness for describing generic {\sc rcf}-expansions, it turns out that the \emph{natural extension} of the one-dimensional dynamical system\footnote{Throughout, $\mathcal{B}$ denotes the Borel $\sigma$-algebra on the appropriate domain.} $([0,1),\mathcal{B},\nu_G,G)$ is better-suited for the study of \emph{{\sc rcf}-approximation coefficients}.  The approximation coefficient of an irrational $x$ and a reduced rational $p/q$, defined by
\begin{equation}\label{approx_coeffs}
\Theta(x,p/q)\coloneqq q^2|x-p/q|,
\end{equation}
indicates how well $x$ is approximated by $p/q$, relative to the size of the denominator $q$.  The limiting distribution of $\Theta(x,p_n/q_n)$---conjectured independently by Doeblin and Lenstra---was obtained in \cite{BJW83} by using the natural extension $(\Omega,\mathcal{B},\bar\nu_G,\mathcal{G})$ of $([0,1),\mathcal{B},\nu_G,G)$ introduced by Nakada in \cite{N1981}.  Here, $\Omega\coloneqq [0,1)\times [0,1]$; $\G:\Omega\to\Omega$ is defined by $\G(z)=z$ if $z=(0,y)$ and
\begin{equation}\label{NE_G}
\G(z)=\left(\frac{1}{x}-\left\lfloor \frac1x\right\rfloor,\frac{1}{\left\lfloor \frac1x\right\rfloor+y}\right)
\end{equation}
for $z=(x,y)$ with $x\neq 0$; and $\bar\nu_G$ is the ergodic, absolutely continuous, $\G$-invariant probability measure with density $1/(\log(2)(1+xy)^2)$.  The natural extension records both the `future' and `past' of $G$-orbits, and $\G$ acts as a two-sided shift on {\sc rcf}-expansions: if $z=(x,y)$ with $x=[0;a_1,a_2,\dots]$ irrational and $y$ has (finite or infinite) {\sc rcf}-expansion $[0;b_1,b_2,\dots]$, then $\G(x,y)=([0;a_2,a_3,\dots],[0;a_1,b_1,b_2,\dots])$.  Geometrically, $\G$ maps the interior of the vertical rectangle $V_a\coloneqq((1/(a+1),1/a]\times [0,1])\cap \Omega$ diffeomorphically to the interior of the horizontal rectangle $H_a\coloneqq[0,1)\times (1/(a+1),1/a]$ for each integer $a\ge 1$; see Figure \ref{GaussNE-fig}.

\begin{figure}[t]
\centering
\includegraphics[width=.3\textwidth]{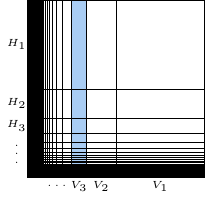}
\hspace{20pt}
\includegraphics[width=.3\textwidth]{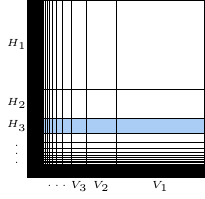}
\caption{The map $\G$ sends the interior of $V_a$ diffeormorphically to the interior of $H_a$.}
\label{GaussNE-fig}
\end{figure}

There are several ways of generalising the notion of {\sc rcf}s introduced above.  In this article, a \emph{generalised continued fraction} ({\sc gcf}) is a formal expression
\begin{equation}\label{gcf}
[\beta_0;\alpha_0/\beta_1,\alpha_1/\beta_2,\dots]=\beta_0+\cfrac{\alpha_0}{\beta_1+\cfrac{\alpha_1}{\beta_2+\ddots}}\ ,
\end{equation}
where the \emph{partial numerators} $\alpha_n$ and \emph{partial denominators} $\beta_n$ are rationals with $\alpha_n\neq 0$.  The \emph{$n^\text{th}$ convergent} of $[\beta_0;\alpha_0/\beta_1,\alpha_1/\beta_2,\dots]$ is the extended rational number
\[\frac{P_n}{Q_n}=[\beta_0;\alpha_0/\beta_1,\dots,\alpha_{n-1}/\beta_n]\coloneqq \beta_0+\cfrac{\alpha_0}{\beta_1+\cfrac{\alpha_1}{\ddots +\cfrac{\alpha_{n-1}}{\beta_n}}}\in \mathbb{Q}\cup\{\infty\}, \quad n\ge 0\]
(when $n=0$, this is $P_0/Q_0=[\beta_0]\coloneqq \beta_0$).  Throughout, we shall assume that $P_n/Q_n\in\mathbb{Q}$ is reduced.  As opposed to {\sc rcf}s, {\sc gcf}s need not converge, i.e., $\lim_{n\to\infty}P_n/Q_n$ might not exist.  If this limit does exist and equals $x\in\mathbb{R}$, we call \eqref{gcf} a \emph{{\sc gcf}-expansion} of $x$ and write $x=[\beta_0;\alpha_0/\beta_1,\alpha_1/\beta_2,\dots]$.  

\begin{remark}\label{remark: equivalent_gcf}
Successively multiplying each triple $\alpha_n,\beta_{n+1},\alpha_{n+1}$ by an appropriate integer, one may obtain an \emph{equivalent} {\sc gcf}---i.e., a {\sc gcf} with identical convergents---whose partial numerators and partial denominators $\alpha_n$, $\beta_{n+1}$, $n\ge 0$, are integers; see \cite{S1855}.  Our use of rational digits is explained in Remark \ref{remark: rational_digits} below.
\end{remark}

In the dynamical theory of continued fractions, several generalisations and adaptations of the Gauss map produce specific types of {\sc gcf}s---called \emph{semi-regular continued fractions} ({\sc srcf}s)---which are known to converge; see, e.g., the parameterised families \cite{K1991,N1981}, which contain many well-studied {\sc srcf}-algorithms.  A {\sc srcf} is a {\sc gcf} as in \eqref{gcf} satisfying (i) $\alpha_n=\pm 1$ for $n\ge 0$, (ii) $\beta_n\in\mathbb{Z}$, $n\ge 0$, with $\beta_n>0$ for $n>0$, and (iii) $\alpha_n+\beta_n\ge 1$ for all $n>0$, with this final inequality strict infinitely often.  It is known that for a.e.\footnote{All \emph{almost every} statements are with respect to Lebesgue measure.}~irrational $x$, any {\sc srcf}-expansion of $x$ with reduced convergents $P_k/Q_k$ satisfies 
\begin{enumerate}
\item[(i)] $\sup_{k\ge 0} \Theta(x,P_k/Q_k)\ge 1/2$ and
\item[(ii)] $\limsup_{k\to\infty} n(k)/k\le \log 2/\log \varphi$,   
\end{enumerate}
where $n(k)$ is such that $q_{n(k)}\le Q_k<q_{n(k)+1}$ with $p_n/q_n$ the $n^{\text{th}}$ {\sc rcf}-convergent of $x$, and $\varphi=(\sqrt{5}+1)/2$ is the golden ratio; see Remarks 2.19 and 3.16 of \cite{B1987}.

In \cite{B1987}, Bosma introduces an algorithm producing so-called \emph{optimal continued fractions}\footnote{Bosma writes that these expansions were originally given by Selenius in \cite{S1960}, where they were constructed using the {\sc rcf}-expansion of $x$.  Bosma's algorithm does not require prior knowledge of the {\sc rcf}-expansion; see also Remark \ref{no_rcf_needed} below.} ({\sc ocf}s) which provide `best approximations' and the `fastest convergence' among {\sc srcf}s: the {\sc ocf}s satisfy both $\Theta(x,P_k/Q_k)<1/2$ for all $k\ge 0$, and $\lim_{k\to\infty} n(k)/k=\log 2/\log \varphi$ a.s.  Note that the {\sc ocf}s are `optimal' (in the sense of (i) and (ii) above) only among {\sc srcf}s, but not necessarily among {\sc gcf}-expansions.  Motivated by these considerations, we make the following:
\begin{defn}\label{defn: socf}
Let $\varepsilon, C>0$.  A {\sc gcf}-expansion of $x$ with reduced convergents $P_k/Q_k$ is called \emph{$(\varepsilon,C)$-superoptimal} if both
\begin{enumerate}
\item[(i)] $\Theta(x,P_k/Q_k)\le \varepsilon$ for all $k\ge 0$ and
\item[(ii)] $\liminf_{k\to\infty} n(k)/k\ge C$, 
\end{enumerate}
where $q_{n(k)}\le Q_k<q_{n(k)+1}$ with $p_n/q_n$ the $n^\text{th}$ {\sc rcf}-convergent of $x$.
\end{defn}

We remark that for small $\varepsilon$, condition (i) says that the convergents $P_k/Q_k$ are `extremely good approximations' to $x$, and for large $C$, condition (ii) says that the convergents approach $x$ `extremely quickly' (relative to {\sc rcf}-convergents).  The goal of this article is to introduce---for any choice of $\varepsilon$ and $C$---a whole family of dynamical systems $(\Delta,\mathcal{B},\bar\nu_\Delta,\G_\Delta)$ parameterised by subregions $\Delta\subset\Omega$, each of which produces an $(\varepsilon,C)$-\emph{superoptimal continued fraction} ({\sc socf}) expansion of a.e.\footnote{For $\varepsilon\ge 1/\sqrt{5}$, one can construct explicit $(\Delta,\mathcal{B},\bar\nu_\Delta,\G_\Delta)$ which produce \emph{for all} irrationals in $(0,1)$ a {\sc gcf}-expansion satisfying condition (i) of Definition \ref{defn: socf}; see \S\ref{Hurwitz--Borel continued fractions} below.  A result of Hurwitz (Theorem \ref{thm: Hurwitz}) necessitates the use of `a.e.' for smaller values of $\varepsilon$; see also Remark \ref{lagrange_spectrum}.}~$x\in(0,1)$; see Theorem \ref{eps_C_so_thm} below.  

Our construction may viewed as a generalisation of that underlying Kraaikamp's \emph{$S$-expansions}, which use certain induced transformations of the natural extension $(\Omega,\mathcal{B},\bar\nu_G,\mathcal{G})$ to govern \emph{singularisations} of {\sc rcf}-expansions (\cite{K1991}).  Singularisation is an old, arithmetic procedure which sometimes allows one to transform a given {\sc srcf}-expansion into a new {\sc srcf}-expansion whose convergents are a subsequence of the original convergents.  In \cite{K1991}, Kraaikamp fixes a measurable \emph{singularisation area} $S\subset\Omega$, considers the $\G$-orbit of $z=(x,0)$, and---via singularisation---produces a {\sc srcf}-expansion of $x$ whose convergents are precisely the subsequence of {\sc rcf}-convergents $p_n/q_n$ for which $\G^{n}(z)\in \Delta_S\coloneqq \Omega\backslash S$.  While Kraaikamp's $S$-expansions define a large collection of well-studied and new continued fraction algorithms---including {\sc ocf}s---they are limited by certain restrictions imposed by singularisation.  In particular, beginning with an {\sc rcf}-expansion, singularisation only allows one to remove $p_n/q_n$ if $a_{n+1}=1$, and two consecutive convergents $p_n/q_n,$ $p_{n+1}/q_{n+1}$ may not be removed.  These two restrictions impose corresponding restrictions on the singularisation areas $S\subset \Omega$, and hence on the regions $\Delta_S$ that one induces on.  

The algorithms underlying our $(\varepsilon,C)$-{\sc socf}s follow the same strategy as $S$-expansions, but replace singularisation with Seidel's more powerful acceleration technique of \emph{contraction} (\cite{S1855}).  Under mild assumptions (which are satisfied for {\sc rcf}s), contraction produces from a given {\sc gcf} a new, explicit {\sc gcf} whose convergents are \emph{any} prescribed subsequence of the original convergents.  Contraction thus relieves us from the aforementioned restrictions imposed on singularisation areas and allows one to use induced transformations of $(\Omega,\mathcal{B},\bar\nu_G,\mathcal{G})$ on \emph{any} positive-measure subregion $\Delta\subset\Omega$ to govern accelerations of {\sc rcf}-expansions.   

\begin{remark}
This coupling of induced transformations with contraction was first applied in \cite{DKS2026} in the setting of Shunji Ito's natural extension of the Farey tent map (\cite{I1989}).  There is an isomorphic copy of the natural extension $(\Omega,\mathcal{B},\bar\nu_G,\mathcal{G})$ sitting within Ito's natural extension (Theorem 1 of \cite{BY1996}), so the methods of the current paper may be applied within the broader framework of \cite{DKS2026}; see section 1.9 of 
\cite{S2025}.  However, a classic result of Legendre (Theorem \ref{thm: Legendre} below) states that $\Theta(x,p/q)<1/2$ implies $p/q$ is some {\sc rcf}-convergent of $x$, so for small $\varepsilon$, the convergents of an $(\varepsilon,C)$-{\sc socf} are bound to be a subsequence of {\sc rcf}-convergents.  Thus we find it natural to describe our {\sc socf}s solely within $(\Omega,\mathcal{B},\bar\nu_G,\mathcal{G})$. 

Applying the same techniques of this paper to the natural extension of the Farey tent map, one could---in addition to the {\sc socf}s produced here---construct for any large $\varepsilon$ and small $C$ an `$(\varepsilon,C)$-suboptimal' continued fraction expansion of a.e.~$x$, where the inequalities in Definition \ref{defn: socf} are reversed and the $\liminf$ is replaced by $\limsup$.
\end{remark}

This article is organised as follows.  We begin in \S\ref{Matrix notation and contraction} by establishing matrix notation and recalling the contraction technique of Seidel.  The induced transformations $(\Delta,\mathcal{B},\bar\nu_\Delta,\G_\Delta)$ are defined in \S\ref{Induced transformations and equidistribution}, and in \S\ref{Inducing contractions of rcfs} we use these to govern contractions of {\sc rcf}-expansions.  Moreover, we show that the partial numerators and partial denominators of the resulting {\sc gcf}s are determined directly from the corresponding system $(\Delta,\mathcal{B},\bar\nu_\Delta,\G_\Delta)$ without needing to know the {\sc rcf}-expansion (Proposition \ref{prop: induced_determines_cf_alg}).  In \S\ref{Superoptimal continued fractions}, we prove our main result (Theorem \ref{eps_C_so_thm}), which gives sufficient criteria on $\Delta$ to guarantee that for a.e.~$x$, the resulting {\sc gcf}-expansion of $x$ is $(\varepsilon,C)$-superoptimal.  We end in \S\ref{Examples} by investigating two specific families of induced systems and the {\sc socf}s they produce.  

\bigskip

{\flushleft \textbf{Acknowledgments.}} The author thanks Karma Dajani, Cor Kraaikamp and the anonymous referee for helpful suggestions which improved this article.  This work is part of project number 613.009.135 of the research programme Mathematics Clusters which is financed by the Dutch Research Council (NWO).

\bigskip

\medskip

\section{Inducing contractions of regular continued fractions}\label{Inducing contractions of regular continued fractions}

\subsection{Matrix notation and contraction}\label{Matrix notation and contraction}

Let $a:(0,1)\to\mathbb{N}$ and $M:(0,1)\to\mathrm{GL}_2\mathbb{Z}$ be defined by 
\begin{equation}\label{a_and_M}
a(x)\coloneqq \lfloor 1/x\rfloor \qquad \text{and} \qquad M(x)\coloneqq \begin{pmatrix}0 & 1\\ 1 & a(x)\end{pmatrix}.
\end{equation}
For $x\in(0,1)$ irrational, set $x_n\coloneqq G^n(x)$ and $a_{n+1}(x)\coloneqq a(x_n)$, $n\ge 0$; when $x$ is understood, we use the suppressed notation $a_n=a_n(x)$, $n>0$.  Moreover, set 
\[M_{n+1}(x)\coloneqq M(x_n)=\begin{pmatrix}0 & 1\\ 1 & a_{n+1}\end{pmatrix}, \quad n\ge 0;\]
similarly here, if $x$ is understood, we write $M_n=M_n(x)$, $n>0$.  

Notice that $x_{n+1}=M_{n+1}^{-1}\cdot x_n,$ $n\ge 0$, where for $M=\left(\begin{smallmatrix}a & b\\ c & d\end{smallmatrix}\right)\in\mathrm{GL}_2\mathbb{Z}$, $M\cdot x=(ax+b)/(cx+d)$ denotes the action of $M$ as a M\"obius transformation.  We thus find $x_n=M_{n+1}\cdot x_{n+1}$, $n\ge 0$, and hence $x=M_1M_2\cdots M_n\cdot x_n$, $n>0$.  For $1\le m\le n$, set
\begin{equation}\label{M_{[m,n]}}
M_{[m,n]}(x)\coloneqq M_m(x)M_{m+1}(x)\cdots M_n(x)=M(x_{m-1})M(x_m)\cdots M(x_{n-1});
\end{equation}
when $x$ is understood, write $M_{[m,n]}=M_{[m,n]}(x)$.  Denote the entries of $M_{[m,n]}(x)$ by 
\[\begin{pmatrix}
r_{[m,n]}(x) & p_{[m,n]}(x)\\
s_{[m,n]}(x) & q_{[m,n]}(x)
\end{pmatrix}
\coloneqq M_{[m,n]}(x), \quad 1\le m\le n,\]
or, when $x$ is understood, $u_{[m,n]}=u_{[m,n]}(x)$ for $u\in \{r,s,p,q\}$.  When $m<n$, one finds from the relation $M_{[m,n]}=M_{[m,n-1]}M_n$ that $r_{[m,n]}=p_{[m,n-1]}$ and $s_{[m,n]}=q_{[m,n-1]}$.  Setting $p_{[m,m-1]}\coloneqq 0$ and $q_{[m,m-1]}\coloneqq 1$, we thus have
\[M_{[m,n]}=\begin{pmatrix}
p_{[m,n-1]} & p_{[m,n]}\\
q_{[m,n-1]} & q_{[m,n]}
\end{pmatrix}, \quad 1\le m\le n.
\]
When $m=1$, we use the special notation $p_n(x)\coloneqq p_{[1,n]}(x)$ and $q_n(x)\coloneqq q_{[1,n]}(x)$, $n\ge 0$, and when $x$ is understood, we set $p_n=p_n(x)$ and $q_n=q_n(x)$.  The digits $a_n$ and rationals $p_n/q_n$ are the partial denominators and the convergents, respectively, of the {\sc rcf}-expansion of $x$ from \eqref{rcf_expn} and \eqref{convergents}.

Recall the definition of a {\sc gcf} from \eqref{gcf}.  While mentioned in, e.g., \cite{LW2008,P1950}, the following definition and theorem---which are central to the remainder of this article---seem to be little-used in full generality; see, e.g., \cite{B2014,DKS2026}.
\begin{defn}\label{defn: contraction}
The \emph{contraction} of an irrational $x=[a_0;a_1,a_2,\dots]$ with respect to a strictly increasing sequence of non-negative integers $(n_k)_{k\ge 0}$ is the {\sc gcf} $[\beta_0;\alpha_0/\beta_1,\alpha_1/\beta_2,\dots]$, with
\[\alpha_{k}=\frac{(-1)^{n_k-n_{k-1}+1}q_{[n_{k-2}+2,n_{k-1}]}}{q_{[n_{k-1}+2,n_k]}}
\qquad \text{and} \qquad 
\beta_k=\frac{q_{[n_{k-2}+2,n_k]}}{q_{[n_{k-1}+2,n_k]}},
\quad k\ge 0,
\]
where $n_k\coloneqq k$ for $k<0$, $q_{[0,-1]}\coloneqq 1$ and $q_{[0,n_0]}\coloneqq p_{[1,n_0]}=p_{n_0}$.
\end{defn}

\begin{thm}[Seidel, 1855 \cite{S1855}]\label{thm: seidel}
Let $x$ be an irrational with {\sc rcf}-convergents $p_n/q_n$.   Then the contraction of $x$ with respect to $(n_k)_{k\ge 0}$ is a {\sc gcf}-expansion of $x$ with convergents $P_k/Q_k=p_{n_k}/q_{n_k}$, $k\ge 0$.
\end{thm}

\begin{remark}
Seidel's theorem may be stated more generally to produce from a given {\sc gcf} (satisfying mild assumptions\footnote{These assumptions---which hold for {\sc rcf}s as considered here---are essentially that the convergents of the original {\sc gcf} are pairwise distinct; see Definition 6.1 and Lemma 6.5 of \cite{DKS2026}.}) a new {\sc gcf} whose convergents are any prescribed subsequence of the original {\sc gcf} convergents.  For this broader statement and a proof, see Definition 6.1\footnote{In \cite{DKS2026}, the partial numerators and partial denominators of a contraction are taken to be integers, so---beginning from a {\sc rcf}---the definition of contraction there and Definition \ref{defn: contraction} here differ.  However, these two definitions produce \emph{equivalent} {\sc gcf}s, i.e., {\sc gcf}s whose convergents are identical; see \cite{S1855} and Remark \ref{remark: equivalent_gcf} above.} and Theorem 6.6 of \cite{DKS2026}.  We also remark that Seidel's result is a consequence of a yet older theorem of Bernoulli (\cite{B1775}) which allows for the construction of a {\sc gcf} with any prescribed sequence of convergents $(P_k/Q_k)_{k\ge 0}$ satisfying $P_k/Q_k\neq P_{k+1}/Q_{k+1}$; see \cite{LW2008,P1950}.
\end{remark}

\subsection{Induced transformations and equidistribution}\label{Induced transformations and equidistribution}

Recall the natural extension $(\Omega,\mathcal{B},\bar\nu_G,\G)$ from \S\ref{Introduction}.  Let $\Delta\subset\Omega$ be $\bar\nu_G$-measurable with $\bar\nu_G(\Delta)>0$, and let $j=j_\Delta:\Omega\to\mathbb{N}\cup\{\infty\}$ denote the \emph{hitting time} to $\Delta$, i.e.,
\[j(z)=j_\Delta(z)\coloneqq \inf\{n>0 \mid \G^n(z)\in \Delta\}.\]
Since $(\Omega,\mathcal{B},\bar\nu_G,\G)$ is ergodic and conservative, $j(z)<\infty$ for $\bar\nu_G$--a.e.~$z\in \Omega$; in what follows, we remove from each measurable set $S\subset \Omega$ the null-set of points that enter $\Delta$ at most finitely often under iterates of $\G$ and denote this new set again by $S$.  Let $\G_\Delta:\Omega\to \Delta$ be the \emph{induced map} given by $\G_\Delta(z)\coloneqq \G^{j(z)}(z)$, and $\bar\nu_\Delta$ the \emph{induced measure} defined by $\bar\nu_\Delta(S)\coloneqq \bar\nu_G(S)/\bar\nu_G(\Delta)$ for any measurable $S\subset \Delta$.  Ergodicity of $(\Omega,\mathcal{B},\bar\nu_G,\G)$ implies that of the \emph{induced system} $(\Delta,\mathcal{B},\bar\nu_\Delta,\G_\Delta)$.  

\begin{remark}\label{remark: entropy}
We remark that the (metric) entropy of $(\Delta,\mathcal{B},\bar\nu_\Delta,\G_\Delta)$ is $h(\G_\Delta)=\pi^2/(6\log(2)\bar\nu_G(\Delta))$, which follows from Abramov's formula (\cite{A1959}) and the fact that the entropy of the Gauss map (and hence also its natural extension) is $h(G)=\pi^2/(6\log(2))$.  
\end{remark}

For $z\in\Omega$, define 
\[z_n^\Delta=(x_n^\Delta,y_n^\Delta)\coloneqq\G_\Delta^n(z) \qquad \text{and} \qquad j_n(z)=j_n^\Delta(z)=\sum_{k=0}^{n-1}j(z_k^\Delta), \quad n\ge 0,\]
where the summation is understood to be zero for $n=0$, i.e., $j_0(z)=0$.  When the point $z$ is understood, we use the suppressed notation $j_n=j_n(z)$, $n\ge 0$.  With this notation, $z_n^\Delta=\G^{j_n}(z)$, i.e., $(j_n)_{n\ge 1}$ is precisely the subsequence of indices $j\ge 1$ for which $\G^j(z)\in \Delta$.

Later, we shall be interested in $\G$-orbits of points of the form $z=(x,0)$, and we should like to make statements about such points for a.e.~$x\in (0,1)$.  While $(0,1)\times\{0\}$ is a $\bar\nu_G$--null set---so the pointwise ergodic theorem is not applicable---for a.e.~$x\in (0,1)$, the $\G$-orbit of $z=(x,0)$ nevertheless behaves like that of a `generic' point in $\Omega$:

\begin{thm}\label{equidist_thm}
For a.e.~$x\in (0,1)$, the $\G$-orbit of $z=(x,0)$ is $\bar\nu_G$-equidistributed.  In particular, 
\[\lim_{n\to\infty}\frac1n\sum_{k=0}^{n-1}\mathbf{1}_\Delta(z_k)=\bar\nu_G(\Delta)\]
for any \emph{$\bar\nu_G$-continuity set} $\Delta\subset\Omega$---i.e., any measurable set whose boundary is null---where $z_k=\G^k(z)$.
\end{thm}
For a proof of this, see \cite{J86} and Theorem 1.2 of Chapter 3 of \cite{KN1974} (see also Remark 4.6.i of \cite{K1991}).

\subsection{Inducing contractions of {\sc rcf}s}\label{Inducing contractions of rcfs}

Let $\Delta\subset\Omega$ be $\bar\nu_G$-measurable with $\bar\nu_G(\Delta)>0$.  When defined,\footnote{I.e., when the $\G$-orbit of $z=(x,0)$ enters $\Delta$ infinitely often.  By Theorem \ref{equidist_thm}, this occurs for a.e.~$x$ if $\Delta$ is a $\bar\nu_G$-continuity set.  More generally---using ergodicity and the fact that $|\G^n(x,0)-\G^n(x,y)|\to 0$ uniformly in $y$---one can show that this occurs a.s.~if the interior of $\Delta$ has positive measure.} denote by 
\[[\beta_0^\Delta(x);\alpha_0^\Delta(x)/\beta_1^\Delta(x),\alpha_1^\Delta(x)/\beta_2^\Delta(x),\dots]\]
the contraction of $x$ with respect to $(n_k)_{k\ge 0}$, where $n_k\coloneqq j_{k+1}-1=j_{k+1}^\Delta(z)-1$ and $z=(x,0)$.  When the point $x$ is understood, it is suppressed from the notation and we write $x=[\beta_0^\Delta;\alpha_0^\Delta/\beta_1^\Delta,\alpha_1^\Delta/\beta_2^\Delta,\dots]$.  By Seidel's theorem (Theorem \ref{thm: seidel}), the $k^\text{th}$ convergent $P_k^\Delta/Q_k^\Delta$ of $[\beta_0^\Delta;\alpha_0^\Delta/\beta_1^\Delta,\alpha_1^\Delta/\beta_2^\Delta,\dots]$ is $P_k^\Delta/Q_k^\Delta=p_{j_{k+1}-1}/q_{j_{k+1}-1}$, $k\ge 0$, where $p_n/q_n$ is the $n^\text{th}$ {\sc rcf}-convergent of $x$.

\begin{remark}\label{remark: indexing}
We choose to index with $n_k=j_{k+1}-1$ rather than $n_k=j_k$ for two reasons: first, it simplifies the examples in \S\ref{Examples} below, and second, it makes the connection to the broader theory of contracted Farey expansions from \cite{DKS2026} immediate (after changing coordinates via $(x,y)\mapsto (x,1/(1+y))$; see \S7.1 of \cite{DKS2026}).  While our indexing differs from \cite{K1991}, for any singularisation area $S$ and irrational $x$, one obtains the $S$-expansion of $x$ as $x=[\beta_0^\Delta;\alpha_0^\Delta/\beta_1^\Delta,\alpha_1^\Delta/\beta_2^\Delta,\dots]$ where $\Delta=\Omega\setminus \G(S)$; see \S7.2 of \cite{DKS2026}.
\end{remark}

Note that by Definition \ref{defn: contraction}, the partial numerators $\alpha_k^\Delta$ and partial denominators $\beta_k^\Delta$ are defined in terms of the numbers $q_{[m,n]}$ which depend on the {\sc rcf}-expansion of $x$.  We wish to determine $\alpha_k^\Delta$ and $\beta_k^\Delta$ solely in terms of the induced system $(\Delta,\mathcal{B},\bar\nu_\Delta,\G_\Delta)$.  From \eqref{NE_G}, \eqref{a_and_M} and \eqref{M_{[m,n]}}, one finds that for any $n> 0$ and $z=(x,y)\in\Omega$,
\begin{equation}\label{G^n_in_M}
\G^n(z)=\left(M_{[1,n]}^{-1}\cdot x,M_{[1,n]}^T\cdot y\right),
\end{equation}
where $M_{[1,n]}=M_{[1,n]}(x)$.  Setting $M_\Delta(z)\coloneqq M_{[1,j(z)]}(x)$, $M_\Delta^{-1}(z)\coloneqq (M_\Delta(z))^{-1}$ and $M_\Delta^T(z)\coloneqq (M_\Delta(z))^T$, we thus have
\begin{equation}\label{G_Delta}
\G_\Delta(z)=\left(M_\Delta^{-1}(z)\cdot x, M_\Delta^T(z)\cdot y\right).
\end{equation}
Denote the entries of $M_\Delta(z)$ by 
\begin{equation*}
\begin{pmatrix}
r_\Delta(z) & p_\Delta(z)\\
s_\Delta(z) & q_\Delta(z)
\end{pmatrix}
\coloneqq 
M_\Delta(z),
\end{equation*}
and define\footnote{Note that $s_\Delta(z)=q_{j(z)-1}(x)\neq 0$ for all $z$ since $j(z)\ge 1$.} $\alpha_\Delta:\Omega\to\mathbb{Q}\setminus \{0\}$ and $\beta_\Delta:\Omega\to\mathbb{Q}$ by 
\begin{equation}\label{alpha_Delta,beta_Delta}
\alpha_\Delta(z)\coloneqq \frac{(-1)^{j(\G_\Delta(z))+1}s_\Delta(z)}{s_\Delta(\G_\Delta(z))} \qquad \text{and} \qquad \beta_\Delta(z)\coloneqq q_\Delta(z)+\frac{s_\Delta(z)r_\Delta(\G_\Delta(z))}{s_\Delta(\G_\Delta(z))}.
\end{equation}

\begin{prop}\label{prop: induced_determines_cf_alg}
The partial numerators and partial denominators of $x=[\beta_0^\Delta;\alpha_0^\Delta/\beta_1^\Delta,\alpha_1^\Delta/\beta_2^\Delta,\dots]$ may be written as
\[\alpha_0^\Delta=\frac{(-1)^{j(z_0^\Delta)+1}}{s_\Delta(z_0^\Delta)}, \qquad \beta_0^\Delta=\frac{r_\Delta(z_0^\Delta)}{s_\Delta(z_0^\Delta)}, \qquad \alpha_{k+1}^\Delta=\alpha_\Delta(z_k^\Delta) \qquad \text{and} \qquad \beta_{k+1}^\Delta=\beta_\Delta(z_k^\Delta), \quad k\ge 0,\]
where $z_k^\Delta=\G_\Delta^k(z)$ with $z=(x,0)$.
\end{prop}
\begin{proof}
By the definition of $[\beta_0^\Delta;\alpha_0^\Delta/\beta_1^\Delta,\alpha_1^\Delta/\beta_2^\Delta,\dots]$ and Definition \ref{defn: contraction}, we have
\begin{equation}\label{alpha^Delta,beta^Delta}
\alpha_k^\Delta=\frac{(-1)^{j_{k+1}-j_k+1}q_{[j_{k-1}+1,j_k-1]}}{q_{[j_k+1,j_{k+1}-1]}} \qquad \text{and} \qquad \beta_k^\Delta=\frac{q_{[j_{k-1}+1,j_{k+1}-1]}}{q_{[j_k+1,j_{k+1}-1]}},
\quad k\ge 0,
\end{equation}
where $j_k\coloneqq k$ for $k<0$, $q_{[0,-1]}\coloneqq 1$, and $q_{[0,j_1-1]}\coloneqq p_{[1,j_1-1]}$.  Notice that 
\begin{multline*}
M_{[j_k+1,j_{k+1}]}=M_{j_k+1}(x)M_{j_k+2}(x)\cdots M_{j_k+j(z_k^\Delta)}(x)=M_1(x_{j_k})M_2(x_{j_k})\cdots M_{j(z_k^\Delta)}(x_{j_k})\\
=M_1(x_k^\Delta)M_2(x_k^\Delta)\cdots M_{j(z_k^\Delta)}(x_k^\Delta)=M_{[1,j(z_k^\Delta)]}(x_k^\Delta)=M_\Delta(z_k^\Delta),
\end{multline*}
so
\begin{equation}\label{entries_M_Delta_vs_M_js}
\begin{pmatrix}
p_{[j_k+1,j_{k+1}-1]} & p_{[j_k+1,j_{k+1}]}\\
q_{[j_k+1,j_{k+1}-1]} & q_{[j_k+1,j_{k+1}]}
\end{pmatrix}
=
\begin{pmatrix}
r_\Delta(z_k^\Delta) & p_\Delta(z_k^\Delta)\\
s_\Delta(z_k^\Delta) & q_\Delta(z_k^\Delta)
\end{pmatrix},
\quad k\ge 0.
\end{equation}

The claim for $\alpha_k^\Delta$ follows from \eqref{alpha_Delta,beta_Delta}, \eqref{alpha^Delta,beta^Delta}, \eqref{entries_M_Delta_vs_M_js}, and the facts that $j_{k+1}-j_{k}=j(z_k^\Delta)$ and $q_{[0,-1]}=1$.  Next, notice that for $k>0$, $q_{[j_{k-1}+1,j_{k+1}-1]}$ is the bottom-left entry of $M_{[j_{k-1}+1,j_{k+1}]}=M_{[j_{k-1}+1,j_k]}M_{[j_k+1,j_{k+1}]}$, and hence
\begin{multline*}
\frac{q_{[j_{k-1}+1,j_{k+1}-1]}}{q_{[j_k+1,j_{k+1}-1]}}=\frac{q_{[j_{k-1}+1,j_k]}q_{[j_k+1,j_{k+1}-1]}+q_{[j_{k-1}+1,j_k-1]}p_{[j_k+1,j_{k+1}-1]}}{q_{[j_k+1,j_{k+1}-1]}}\\
=q_{[j_{k-1}+1,j_k]}+\frac{q_{[j_{k-1}+1,j_k-1]}p_{[j_k+1,j_{k+1}-1]}}{q_{[j_k+1,j_{k+1}-1]}}.
\end{multline*}
The claim for $\beta_k^\Delta$ follows from this, \eqref{alpha_Delta,beta_Delta}, \eqref{alpha^Delta,beta^Delta}, \eqref{entries_M_Delta_vs_M_js}, $q_{[0,j_1-1]}=p_{[1,j_1-1]}$ and $q_{[0,-1]}=1$.
\end{proof}

\begin{remark}\label{no_rcf_needed}
The upshot of Proposition \ref{prop: induced_determines_cf_alg} is that no prior knowledge of the {\sc rcf}-expansion of $x$ is needed in order to compute the partial numerators and partial denominators of $x=[\beta_0^\Delta;\alpha_0^\Delta/\beta_1^\Delta,\alpha_1^\Delta/\beta_2^\Delta,\dots]$---these digits are obtained directly from the dynamical system $(\Delta,\mathcal{B},\bar\nu_\Delta,\G_\Delta)$.  Indeed, one computes $x=[\beta_0^\Delta;\alpha_0^\Delta/\beta_1^\Delta,\alpha_1^\Delta/\beta_2^\Delta,\dots]$ by determining subregions of $\Delta$ on which the matrices $M_\Delta(z)$ are constant and considering the subsequent subregions\footnote{As seen in \S\ref{Examples}, these regions also \emph{implicitly} determine digits of the {\sc rcf}-expansion of $x$, but for our algorithms no knowledge of the {\sc rcf}-expansion is needed \emph{beforehand}.} entered under the forward $\G_\Delta$-orbit of $z=(x,0)$.  This process is illustrated further in \S\ref{Examples}. 
\end{remark}

\begin{remark}\label{remark: rational_digits}
By Proposition \ref{prop: induced_determines_cf_alg}, the partial numerators and partial denominators $\alpha_k^\Delta$ and $\beta_k^\Delta$ depend on at most two points in the $\G_\Delta$-orbit of $z$.  If one insists---as described in Remark \ref{remark: equivalent_gcf}---that the partial numerators and partial denominators be integers, one finds that they depend on three points in the $\G_\Delta$-orbit of $z$.
\end{remark}

\subsection{Superoptimal continued fractions}\label{Superoptimal continued fractions}

In this subsection, we obtain---for any choice of $\varepsilon$ and $C$---a whole family of induced systems $(\Delta,\mathcal{B},\bar\nu_\Delta,G_\Delta)$ such that $x=[\beta_0^\Delta;\alpha_0^\Delta/\beta_1^\Delta,\alpha_1^\Delta/\beta_2^\Delta,\dots]$ is $(\varepsilon,C)$-superoptimal for a.e.~$x\in(0,1)$.  To begin, let $g:\Omega\to [0,1]$ be defined for $z=(x,y)$ by
\begin{equation}\label{g}
g(z)\coloneqq\frac{y}{1+xy}.
\end{equation}
The following result is well-known, but we include a proof for completeness: 
\begin{prop}\label{Theta_and_g}
For any $n>0$,
\[\Theta(x,p_{n-1}/q_{n-1})=g(z_n),\]
where $z_n=\G^n(z)$ with $z=(x,0)$.
\end{prop}
\begin{proof}
Let $z_n=(x_n,y_n)$, and notice from \eqref{G^n_in_M} that $x=M_{[1,n]}\cdot x_n=(p_{n-1}x_n+p_n)/(q_{n-1}x_n+q_n)$ and $y_n=M_{[1,n]}^T\cdot 0=q_{n-1}/q_n$.  Using \eqref{approx_coeffs}, the fact that $\det M_{[1,n]}=(-1)^n$, and \eqref{g}, we compute
\[\Theta(x,p_{n-1}/q_{n-1})=q_{n-1}^2\left|x-\frac{p_{n-1}}{q_{n-1}}\right|=q_{n-1}^2\left|\frac{p_{n-1}x_n+p_n}{q_{n-1}x_n+q_n}-\frac{p_{n-1}}{q_{n-1}}\right|=\frac{q_{n-1}}{q_{n-1}x_n+q_n}=g(z_n).\qedhere\]
\end{proof}

\begin{thm}\label{eps_C_so_thm}
Let $\varepsilon,C>0$, and let $\Delta$ be a $\bar\nu_G$-continuity set satisfying
\[\text{(a)} \quad \Delta\subset g^{-1}([0,\varepsilon]) \qquad \text{and} \qquad \text{(b)} \quad 0<\bar\nu_G(\Delta)\le 1/C.\]
Then $x=[\beta_0^\Delta;\alpha_0^\Delta/\beta_1^\Delta,\alpha_1^\Delta/\beta_2^\Delta,\dots]$ is $(\varepsilon,C)$-superoptimal for a.e.~$x\in(0,1)$.
\end{thm}
\begin{proof}
Let $x$ belong to the set of full measure satisfying Theorem \ref{equidist_thm}.  By construction, the $k^\text{th}$ convergent of $x=[\beta_0^\Delta/\alpha_0^\Delta;\alpha_1^\Delta/\beta_1^\Delta,\alpha_2^\Delta/\beta_2^\Delta,\dots]$ is $P_k^\Delta/Q_k^\Delta=p_{j_{k+1}-1}/q_{j_{k+1}-1}$.  By Proposition \ref{Theta_and_g} and assumption (a), $\Theta(x,P_k^\Delta/Q_k^\Delta)=g(z_{j_{k+1}})=g(z_{k+1}^\Delta)\le \varepsilon,$ so condition (i) of Definition \ref{defn: socf} is satisfied.  Moreover, 
\[k=\sum_{\ell=1}^{j_{k+1}-1}\mathbf{1}_\Delta(z_\ell)\]
for all $k\ge 1$.  By Theorem \ref{equidist_thm} and assumption (b), dividing both sides by $j_{k+1}-1$ and letting $k\to\infty$ gives $\lim_{k\to\infty}k/(j_{k+1}-1)=\bar\nu_G(\Delta)\le 1/C$.  But $n(k)=j_{k+1}-1$ with $n(k)$ as in (ii) of Definition \ref{defn: socf}, so this condition also holds. 
\end{proof}

Recall from Remark \ref{remark: entropy} the values of the metric entropy $h(G)$ and $h(\G_\Delta)$. We end this subsection by noting that classical results of L\'evy from 1936 are easily extended to our {\sc socf}s (this is shown more generally in Theorem 5.29 of \cite{DKS2025}, but the proof simplifies in our current setting).  Indeed, L\'evy proved\footnote{L\'evy used probabilistic arguments, but these results may also be proven using ergodic theory; see \cite{DK2002B}.} that for a.e.~$x$,
\begin{equation}\label{levy1}
\lim_{n\to\infty} \frac1n\log q_n=\frac12h(G),
\end{equation}
and from this, one also obtains that for a.e.~$x$, 
\begin{equation}\label{levy2}
\lim_{n\to\infty} \frac1n\log\left|x-\frac{p_n}{q_n}\right|=-h(G).
\end{equation}

\begin{cor}
Let $\Delta$ be a $\bar\nu_G$-continuity set with $\bar\nu_G(\Delta)>0$.  Then for a.e.~$x\in (0,1)$, the convergents $P_k^\Delta/Q_k^\Delta$ of $x=[\beta_0^\Delta;\alpha_0^\Delta/\beta_1^\Delta,\alpha_1^\Delta/\beta_2^\Delta,\dots]$ satisfy
\[\text{(i)}\quad \lim_{k\to\infty} \frac1k \log Q_k^\Delta=\frac12h(\G_\Delta)
\qquad \text{and}\qquad \text{(ii)}\quad \lim_{k\to\infty} \frac1k \log\left|x-\frac{P_k^\Delta}{Q_k^\Delta}\right|=-h(\G_\Delta).
\]
\end{cor}
\begin{proof}
Suppose $x$ belongs to the set of full measure satisfying Theorem \ref{equidist_thm}, \eqref{levy1} and \eqref{levy2}.  The claims follow by writing $P_k^\Delta=p_{j_{k+1}-1}$ and $Q_k^\Delta=q_{j_{k+1}-1}$; multiplying both quantities in the limits by $1=(j_{k+1}-1)/(j_{k+1}-1)$; and using \eqref{levy1}, \eqref{levy2} and the fact---as in the proof of Theorem \ref{eps_C_so_thm}---that $\lim_{k\to\infty}(j_{k+1}-1)/k=1/\bar\nu_G(\Delta)$.
\end{proof}

\section{Examples}\label{Examples}

We now turn our attention to two specific families of induced systems $(\Delta,\mathcal{B},\bar\nu_\Delta,\G_\Delta)$ and the {\sc socf}-expansions they produce.  The first family arises from jump transformations of the one-dimensional system $([0,1),\mathcal{B},\nu_G,G)$, and the second family is based on a classical result of Legendre regarding {\sc rcf}-convergents and the approximation coefficients defined in \eqref{approx_coeffs}.

\subsection{Jump continued fractions}\label{Jump continued fractions}

Fix some integer $b>1$, and let $\Delta=[0,1)\times [0,1/b]$.  Using Theorem \ref{eps_C_so_thm}, one finds that $x=[\beta_0^\Delta;\alpha_0^\Delta/\beta_1^\Delta,\alpha_1^\Delta/\beta_2^\Delta,\dots]$ is $(1/b,\log2/\log(1+1/b))$-superoptimal for a.e.~$x\in (0,1)$.  In order to better understand the system $(\Delta,\mathcal{B},\bar\nu_\Delta,\G_\Delta)$ and the {\sc socf}s it produces, we introduce some symbolic notation.  For $W=w_1w_2\cdots$ a finite or an infinite word in the alphabet of positive integers, let $[0;W]\coloneqq [0;w_1,w_2,\dots]$. If $W=w_1\cdots w_n$ is finite, let $|W|\coloneqq n$ denote the length of $W$, and let $p(W)$ and $q(W)$ be the unique, non-negative and relatively prime integers for which $p(W)/q(W)=[0;W]$.  We extend these notions to the empty word $\epsilon$ by setting $[0;\epsilon]\coloneqq 0$, $|\epsilon|\coloneqq 0$, $p(\epsilon)\coloneqq 0$ and $q(\epsilon)\coloneqq 1$.

Now, let 
\[\mathcal{W}_b\coloneqq \{1,\dots,b-1\}^*=\{\epsilon\}\cup\bigcup_{n\ge 1}\{1,\dots,b-1\}^n\]
denote the set of all finite words in the alphabet $\{1,\dots,b-1\}$, including the empty word.  For each $W\in\mathcal{W}_b$, let
\begin{equation}\label{Delta_W}
\Delta_W\coloneqq \{z=(x,y)\in\Delta\ |\ x=[0;Wa\cdots],\ a\ge b\}=\begin{cases}
\Big(\frac{p(W)}{q(W)},\frac{p(Wb)}{q(Wb)}\Big]\times[0,1/b], & |W|\ \text{even},\\
\vspace{-10pt}\\
\Big[\frac{p(Wb)}{q(Wb)},\frac{p(W)}{q(W)}\Big)\times[0,1/b], & |W|\  \text{odd},
\end{cases}
\end{equation}
and for each integer $a\ge b$, let $\Delta_W(a)\subset\Delta_W$ be defined by
\begin{equation}\label{Delta_W(a)}
\Delta_W(a)\coloneqq\{z=(x,y)\in\Delta\ |\ x=[0;Wa\cdots]\}=\begin{cases}
\Big(\frac{p(W(a+1))}{q(W(a+1))},\frac{p(Wa)}{q(Wa)}\Big]\times[0,1/b], & |W|\ \text{even},\\
\vspace{-10pt}\\
\Big[\frac{p(Wa)}{q(Wa)},\frac{p(W(a+1))}{q(W(a+1))}\Big)\times[0,1/b], & |W|\  \text{odd}.
\end{cases}
\end{equation}
Using the fact that $\Delta$ consists of all points $z=(x,y)\in\Omega$ for which $y=[0;W]$, where $W=\epsilon$ or $W=w_1w_2\cdots$ with $w_1\ge b$, we find that for $z\in\Delta_W(a)$,
\begin{equation}\label{G_Delta,M_Delta}
\G_\Delta(z)=\G^{|W|+1}(z) \qquad \text{and} \qquad M_\Delta(z)=\begin{pmatrix}
r_\Delta(z) & p_\Delta(z)\\
s_\Delta(z) & q_\Delta(z)
\end{pmatrix}=\begin{pmatrix}
p(W) & p(Wa)\\
q(W) & q(Wa)
\end{pmatrix}.
\end{equation}
Recalling the functions $\alpha_\Delta$ and $\beta_\Delta$ from \eqref{alpha_Delta,beta_Delta}, it follows that
\begin{equation}\label{alpha_Delta,beta_Delta-fib}
\alpha_\Delta(z)=\frac{(-1)^{|W'|}q(W)}{q(W')}, \quad \beta_\Delta(z)=q(Wa)+\frac{q(W)p(W')}{q(W')}, 
\end{equation}
where $W,W'\in \mathcal{W}_b$ and $a\ge b$ are the unique words and integer, respectively, for which $z\in\Delta_W(a)\cap\G_\Delta^{-1}(\Delta_{W'})$.

\begin{remark}
Notice from \eqref{G_Delta} and \eqref{G_Delta,M_Delta} that the map $\G_\Delta(z)$ depends only on the first coordinate of $z$.  We thus obtain a factor $([0,1),\mathcal{B},\nu_\Delta,G_\Delta)$ of $(\Delta,\mathcal{B},\bar\nu_\Delta,\G_\Delta)$ through the projection map $\pi:\Delta\to[0,1)$, $z=(x,y)\mapsto x$, satisfying $G_\Delta\circ\pi=\pi\circ\G_\Delta$, and one computes that the pushforward measure $\nu_\Delta\coloneqq \pi_*(\bar\nu_\Delta)$ has density $1/(\log(1+1/b)(b+x))$.  In fact, one can show that $G_\Delta$ is the \emph{jump transformation} of the Gauss map $G$ associated to the interval $[0,1/b]$; see \S11.4 of \cite{DK2021}.  The graph of the map $G_\Delta$---which is obtained explicitly from \eqref{G_Delta} and \eqref{M_Delta-fib}--\eqref{Delta_W(a)-fib} below through the projection $\pi$---is shown for $b=2$ in Figure \ref{fibonacci-1d_fig}.

This construction can be reversed and generalised: given some (forward) $\nu_G$-measurable $A\subset[0,1)$ with $\nu_G([0,1)\setminus G(A))=0$, one finds that the jump transformation $J_A:[0,1)\to[0,1)$ of $G$ associated to $A$ satisfies $J_A\circ \pi=\pi\circ \G_\Delta$, where $\Delta=\G(A\times [0,1])$.  Moreover, $\nu_A\coloneqq \pi_*(\bar\nu_\Delta)$ is the ergodic, absolutely continuous invariant probability measure for $J_A$.  Using this and the construction from \S\ref{Inducing contractions of rcfs}, it follows that each jump transformation of the Gauss map gives rise to a one-dimensional {\sc gcf}-algorithm.  In \S\ref{Legendre continued fractions} below, we shall see examples where $\G_\Delta(z)$ depends on both coordinates of $z$, and hence---as in the case of the {\sc ocf} (see Remark 7.3 of \cite{K1991})---there is \emph{not} necessarily an underlying one-dimensional system.  
\end{remark}

\begin{figure}[t]
\centering
\includegraphics[width=.45\textwidth]{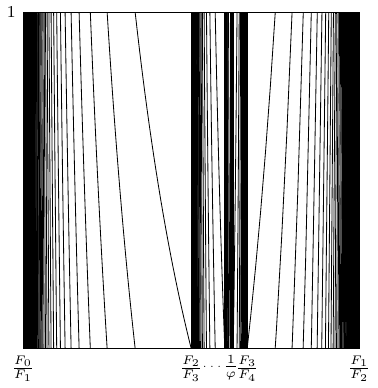}
\caption{The graph of $G_\Delta:[0,1)\to[0,1)$ when $\Delta=[0,1)\times [0,1/2]$.  Here $F_n$ is the $n^\text{th}$ Fibonacci number.}
\label{fibonacci-1d_fig}
\end{figure}

By Proposition \ref{prop: induced_determines_cf_alg}, the digits of the {\sc socf}s $x=[\beta_0^\Delta;\alpha_0^\Delta/\beta_1^\Delta,\alpha_1^\Delta/\beta_2^\Delta,\dots]$ are determined using \eqref{alpha_Delta,beta_Delta-fib} and the subsequent regions $\Delta_W(a)$ entered by the forward $\G_\Delta$-orbit of $z=(x,0)$.  We illustrate further in the specific case that $b=2$, 
which produces $(1/2,\log2/\log(3/2))$-{\sc socf}s.\footnote{Note that $\log 2/\log(3/2)\approx 1.709511\ldots>\log 2/\log\varphi\approx 1.440420\dots$, so a typical {\sc socf}-expansion $x=[\beta_0^\Delta;\alpha_0^\Delta/\beta_1^\Delta,\alpha_1^\Delta/\beta_2^\Delta,\dots]$ converges faster than the {\sc ocf}-expansion of $x$.}

\subsubsection{Fibonacci continued fractions}\label{Fibonacci continued fractions}

Let $b=2$, and note that $\mathcal{W}_2$ consists of the empty word $\epsilon$ and all finite strings of the digit $1$.  For each $W\in\mathcal{W}_2$, set $\Delta_n\coloneqq \Delta_W$ and $\Delta_n(a)\coloneqq \Delta_W(a)$, where $n=|W|$.  One finds that for $z\in\Delta_n(a)$,
\begin{equation}\label{M_Delta-fib}
M_\Delta(z)=\begin{pmatrix}
r_\Delta(z) & p_\Delta(z)\\
s_\Delta(z) & q_\Delta(z)
\end{pmatrix}=\begin{pmatrix}
p(W) & p(Wa)\\
q(W) & q(Wa)
\end{pmatrix}=\begin{pmatrix}
0 & 1\\
1 & 1
\end{pmatrix}^n\begin{pmatrix}
0 & 1\\
1 & a
\end{pmatrix}=
\begin{pmatrix}
F_n & aF_n+F_{n-1}\\
F_{n+1} & aF_{n+1}+F_n
\end{pmatrix},
\end{equation}
where $F_k$ is the $k^\text{th}$ Fibonacci number, i.e., $F_k=F_{k-1}+F_{k-2}$, $k\ge 1$, with $F_{-1}\coloneqq 1$ and $F_0\coloneqq 0$.  Using this, \eqref{Delta_W}--\eqref{G_Delta,M_Delta} and the fact that $2F_k+F_{k-1}=F_{k+2}$, we have 
\begin{equation}\label{Delta_n-fib}
\Delta_n=\begin{cases}
\Big(\frac{F_n}{F_{n+1}},\frac{F_{n+2}}{F_{n+3}}\Big]\times[0,1/2], & n\ \text{even},\\
\vspace{-10pt}\\
\Big[\frac{F_{n+2}}{F_{n+3}},\frac{F_n}{F_{n+1}}\Big)\times[0,1/2], & n\  \text{odd},
\end{cases}
\end{equation}
and 
\begin{align}\label{Delta_W(a)-fib}
\Delta_n(a)&=\begin{cases}
\Big(\frac{(a+1)F_n+F_{n-1}}{(a+1)F_{n+1}+F_{n}},\frac{aF_n+F_{n-1}}{aF_{n+1}+F_{n}}\Big]\times[0,1/2], & n\ \text{even},\\
\vspace{-10pt}\\
\Big[\frac{aF_n+F_{n-1}}{aF_{n+1}+F_{n}},\frac{(a+1)F_n+F_{n-1}}{(a+1)F_{n+1}+F_{n}}\Big)\times[0,1/2], & n\  \text{odd},
\end{cases}\notag\\
&=\left\{z=(x,y)\in \Delta_n\ \middle|\ \left\lfloor\frac{F_{n-1}-F_nx}{F_{n+1}x-F_n}\right\rfloor=a\right\};
\end{align}
see Figure \ref{fibonacci_fig}.

\begin{figure}[t]
\centering
\includegraphics[width=.45\textwidth]{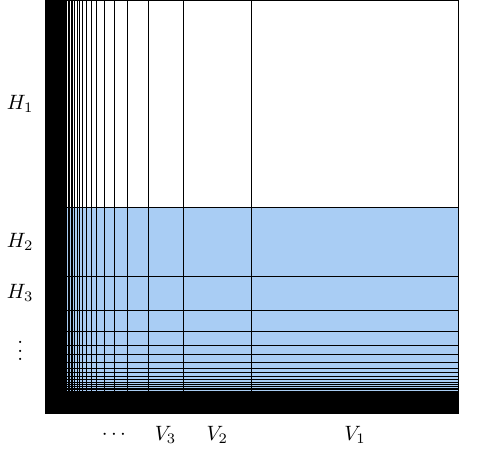}
\includegraphics[width=.45\textwidth]{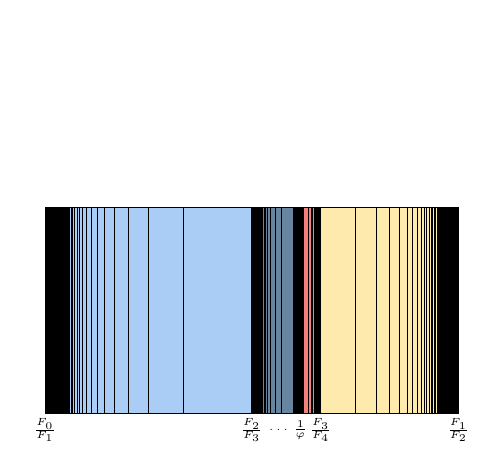}
\caption{Left: The region $\Delta=[0,1)\times[0,1/2]\subset \Omega$.  Right: The regions $\Delta_n(a)\subset \Delta$, indicated by light blue and dark blue for $n$ even and tan and red for $n$ odd.}
\label{fibonacci_fig}
\end{figure}

Recall that the {\sc socf}-expansion $[\beta_0^\Delta;\alpha_0^\Delta/\beta_1^\Delta,\alpha_1^\Delta/\beta_2^\Delta,\dots]$ of an irrational $x$ is defined if and only if the $\G$-orbit of $z=(x,0)$ enters $\Delta$ infinitely often.  In this case, this occurs if and only if the tail of the {\sc rcf}-expansion of $x$ does not end only in $1$'s, or, equivalently, $x\notin G^{-n}(\{1/\varphi\})$ for any $n\ge 0$, where $\varphi$ is the golden ratio.  In particular, $x=[\beta_0^\Delta;\alpha_0^\Delta/\beta_1^\Delta,\alpha_1^\Delta/\beta_2^\Delta,\dots]$ is defined for any transcendental $x$.

\begin{eg}\label{eg: jump}
Let $z=(x,0)$ with $x=\pi-3$.  Using \eqref{G_Delta} and \eqref{M_Delta-fib}--\eqref{Delta_W(a)-fib}, one computes the subsequent regions $\Delta_n(a)$ entered by the $\G_\Delta$-orbit $z_k^\Delta=\G_\Delta^k(z)$ of $z$:
\begin{align*}
z_0^\Delta&\in \Delta_0(7), & z_1^\Delta&\in \Delta_0(15), & z_2^\Delta&\in \Delta_1(292), & z_3^\Delta&\in \Delta_3(2), & z_4^\Delta&\in \Delta_1(3), & z_5^\Delta&\in \Delta_1(14),\\
z_6^\Delta&\in \Delta_0(2), & z_7^\Delta&\in \Delta_2(2), & z_8^\Delta&\in \Delta_0(2), & z_9^\Delta&\in \Delta_0(2), & z_{10}^\Delta&\in \Delta_0(2), & z_{11}^\Delta&\in \Delta_1(84),\dots.
\end{align*}
Using this, Proposition \ref{prop: induced_determines_cf_alg}, \eqref{alpha_Delta,beta_Delta-fib} and \eqref{M_Delta-fib}, we compute
\begin{multline*}
x=[\beta_0^\Delta;\alpha_0^\Delta/\beta_1^\Delta,\alpha_1^\Delta/\beta_2^\Delta,\dots]\\
=[0;1/7,1/16,-1/(881/3),(-1/3)/11,-3/5,-1/15,1/(5/2),(1/2)/5,2/2,1/2,1/3,\dots],
\end{multline*}
which has convergents
\begin{multline*}
\left(\frac{P_k^\Delta}{Q_k^\Delta}\right)_{k\ge 0}=\Bigg(\frac01,\frac17,\frac{16}{113},\frac{14093}{99532},\frac{51669}{364913},\frac{244252}{1725033},\frac{3612111}{25510582},\frac{18549059}{131002976},\\
\frac{48178703}{340262731},\frac{114906465}{811528438},\frac{277991633}{1963319607},\frac{948881364}{6701487259},\dots\Bigg).
\end{multline*}
The corresponding approximation coefficients are plotted in Figure \ref{fig: approx_coefs}.
\end{eg}

\begin{figure}[t]
\centering
\includegraphics[width=.5\textwidth]{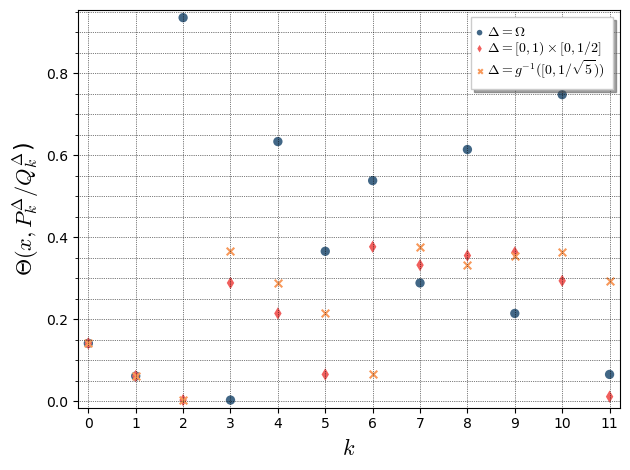}
\caption{The first several approximation coefficients of $x=\pi-3$ and $P_k^\Delta/Q_k^\Delta$ are plotted for $\Delta=\Omega$ (i.e., the {\sc rcf}-expansion), $\Delta=[0,1)\times[0,1/2]$ (from Example \ref{eg: jump}) and $\Delta= g^{-1}([0,1/\sqrt{5}))$ (from Example \ref{eg: hurwitz}).}
\label{fig: approx_coefs}
\end{figure}

\subsection{Legendre continued fractions}\label{Legendre continued fractions}

We begin this subsection by recalling a classical result of Legendre in the theory of {\sc rcf}s.
\begin{thm}[Legendre, 1798 \cite{Le1798}]\label{thm: Legendre}
If $x$ is irrational and $p/q$ is some reduced rational satisfying $\Theta(x,p/q)<1/2$, then $p/q$ is a {\sc rcf}-convergent of $x$, i.e., $p/q=p_n/q_n$ for some $n\ge 0$.  Moreover, the constant $1/2$ is optimal.  
\end{thm}
That the constant $1/2$ is optimal means that for any $\varepsilon>1/2$, there exists some irrational $x$ and some reduced $p/q$ satisfying $\Theta(x,p/q)<\varepsilon$, but $p/q$ is \emph{not} a {\sc rcf}-convergent of $x$.  Now fix $0<\varepsilon_0\le 1/2$, and let $\Delta=g^{-1}([0,\varepsilon_0))$ with $g$ is as in \eqref{g}.  An immediate consequence of Theorems \ref{eps_C_so_thm} and \ref{thm: Legendre} and a computation of $\bar\nu_G(\Delta)=\varepsilon_0/\log2$ is the following:

\begin{cor}\label{leg_socf}
With $\Delta$ as above, $x=[\beta_0^\Delta;\alpha_0^\Delta/\beta_1^\Delta,\alpha_1^\Delta/\beta_2^\Delta,\dots]$ is $(\varepsilon_0,\log2/\varepsilon_0)$-superoptimal for a.e.~$x\in(0,1)$, and the convergents $P_k^\Delta/Q_k^\Delta$ are precisely the rationals $p/q$---ordered with increasing denominators---for which $\Theta(x,p/q)<\varepsilon_0$.
\end{cor}

\begin{remark}
In \cite{K1989}, Kraaikamp shows that Minkowski's \emph{diagonal continued fraction} is an $S$-expansion, and the corresponding two-dimensional dynamical system is explicitly determined.  Up to the change of indices from Remark \ref{remark: indexing}, one finds that this system is $(\Delta,\mathcal{B},\bar\nu_\Delta,\G_\Delta)$ where $\Delta=g^{-1}([0,1/2))$.  Hence for $\varepsilon_0<1/2$, the {\sc socf}s of Corollary \ref{leg_socf} may be viewed as improvements of the diagonal continued fraction.
\end{remark}

\begin{remark}\label{lagrange_spectrum}
In general, $x=[\beta_0^\Delta;\alpha_0^\Delta/\beta_1^\Delta,\alpha_1^\Delta/\beta_2^\Delta,\dots]$ is only defined on a full-measure subset of $(0,1)$, i.e., on the set of irrational $x$ for which the $\G$-orbit of $z=(x,0)$ enters $\Delta$ infinitely often.  We remark that for $\Delta$ as above, the set of $x$ for which $x=[\beta_0^\Delta;\alpha_0^\Delta/\beta_1^\Delta,\alpha_1^\Delta/\beta_2^\Delta,\dots]$ is defined is related to the \emph{Lagrange spectrum} $\mathcal{L}=\{L(x)\ |\ x\in\mathbb{R}\backslash\mathbb{Q}\}$, where the \emph{Lagrange number} $L(x)$ is the supremum over all $L>0$ for which $\Theta(x,p/q)<1/L$ holds infinitely often.  For $\varepsilon_0$ large enough, one can guarantee that $x=[\beta_0^\Delta;\alpha_0^\Delta/\beta_1^\Delta,\alpha_1^\Delta/\beta_2^\Delta,\dots]$ is defined \emph{for all} irrational $x$.  In what follows, we consider the smallest value of $\varepsilon_0$ for which this is possible and determine the underlying system $(\Delta,\mathcal{B},\bar\nu_\Delta,\G_\Delta)$ explicitly.  
\end{remark}

\subsubsection{Hurwitz--Borel continued fractions}\label{Hurwitz--Borel continued fractions}

In 1891, Hurwitz proved the following:
\begin{thm}[Hurwitz, 1891 \cite{H1891}]\label{thm: Hurwitz}
For any irrational $x$, there exist infinitely many reduced rationals $p/q$ for which $\Theta(x,p/q)<1/\sqrt{5}$.  The constant $1/\sqrt{5}$ is optimal.
\end{thm}
Here, the constant $1/\sqrt{5}$ is optimal in the sense that for any $\varepsilon<1/\sqrt{5}$, there exist irrational $x$ (e.g., $x=1/\varphi=[0;1,1,1,\dots]$) for which $\Theta(x,p/q)<\varepsilon$ holds for only finitely many $p/q$.  The first statement of Hurwitz's theorem is a simple consequence of the following, later result of Borel:
\begin{thm}[Borel, 1903 \cite{B1903}]\label{thm: Borel}
For any irrational $x$ and any $n\ge 0$, at least one of $\Theta(x,p_{n+j}/q_{n+j})$, $j\in\{0,1,2\}$, is less than $1/\sqrt{5}$.
\end{thm}

Setting $\Delta=g^{-1}([0,1/\sqrt{5}))$ (see the left-hand side of Figure \ref{hurwitz_fig}), we obtain from Proposition \ref{Theta_and_g} and Theorems \ref{eps_C_so_thm}, \ref{thm: Legendre} and \ref{thm: Hurwitz} the following:
\begin{cor}\label{hurwitz_socf}
With $\Delta$ as above, $x=[\beta_0^\Delta;\alpha_0^\Delta/\beta_1^\Delta,\alpha_1^\Delta/\beta_2^\Delta,\dots]$ is defined \emph{for all} irrational $x\in (0,1)$ and is a.s.~$(1/\sqrt{5},\sqrt{5}\log2)$-superoptimal.  Moreover, the convergents $P_k^\Delta/Q_k^\Delta$ are precisely the rationals $p/q$---ordered with increasing denominators---satisfying Theorem \ref{thm: Hurwitz}.
\end{cor}

We consider the {\sc socf}s from Corollary \ref{hurwitz_socf} in further detail.  To do so, we wish to partition $\Delta=g^{-1}([0,1/\sqrt{5}))$ into subregions on which the matrices $M_\Delta(z)$ are constant, and for this, we need the following:

\begin{prop}\label{g(z_n)}
For any $z=(x,y)\in\Omega$ and $n>0$,
\[g(z_n)=(-1)^n\frac{(p_{n-1}-q_{n-1}x)(p_{n-1}y+q_{n-1})}{1+xy},\]
where $z_n=\G^n(z)$, $p_{n-1}=p_{n-1}(x)$ and $q_{n-1}=q_{n-1}(x)$.  
\end{prop}
\begin{proof}
The claim follows from a straightforward calculation using the definition of $g$ and the facts that 
\[z_n=\left(M_{[1,n]}^{-1}\cdot x,M_{[1,n]}^T\cdot y\right)=\left(\frac{q_nx-p_n}{p_{n-1}-q_{n-1}x},\frac{p_{n-1}y+q_{n-1}}{p_ny+q_n}\right)\]
and $\det M_{[1,n]}=(-1)^n$.
\end{proof}

\begin{remark}
We briefly remark that using the previous proposition, Proposition \ref{Theta_and_g}, and the geometry of the image of a general region $\Delta\subset\Omega$ under the map $z\mapsto (g(z_j^\Delta))_{j=0}^{k-1}\in[0,1]^k$, one can obtain information about $k$ consecutive approximation coefficients $\Theta(x,P_{n+j}^\Delta/Q_{n+j}^\Delta)$, $j\in\{0,\dots,k-1\}$.  For more on this (in the broader context of contracted Farey expansions) and its relation to classical results of, e.g., Vahlen (\cite{V95}), see \S5.3 of \cite{DKS2025} and \S1.8.5 of \cite{S2025}.
\end{remark}

Now for each $n>0$, let $\Delta_n=\Delta\cap j^{-1}(\{n\})$ be the subset of points $z\in\Delta$ whose hitting time is $j(z)=n$.  By Proposition \ref{g(z_n)},
\begin{multline*}
\Delta_n=\{z\in\Delta\setminus \cup_{k=1}^{n-1}\Delta_k\ |\ g(z_n)<1/\sqrt{5}\}\\
=\left\{z\in\Delta\setminus \cup_{k=1}^{n-1}\Delta_k\ \Big|\ (-1)^n\frac{(p_{n-1}-q_{n-1}x)(p_{n-1}y+q_{n-1})}{1+xy}<1/\sqrt{5}\right\}.
\end{multline*}
Since for any $x\in [0,1)$ we have $p_0=0$ and $q_0=1$, we find
\begin{equation}\label{Delta_1}
\Delta_1=\left\{z\in \Delta\ \Big|\ \frac{x}{1+xy}<1/\sqrt{5}\right\}=\left\{z\in \Delta\ \Big|\ x<\frac{1}{\sqrt{5}-y}\right\}.
\end{equation}
Note that $\Delta_2\subset \Delta\setminus \Delta_1$ is contained in the union of vertical rectangles $V_1\cup V_2$; if $z\in V_1$, then $p_1=1$ and $q_1=1$, so
\begin{equation}\label{Delta_2,1}
\Delta_{2,1}\coloneqq \Delta_2\cap V_1=\left\{z\in (\Delta\setminus \Delta_1)\cap V_1\ \Big|\ \frac{(1-x)(y+1)}{1+xy}<1/\sqrt{5}\right\}=\left\{z\in \Delta \ \Big|\ \frac{\sqrt{5}y+\sqrt{5}-1}{(\sqrt{5}+1)y+\sqrt{5}}<x\right\},
\end{equation}
while if $z\in V_2$, then $p_1=1$ and $q_1=2$, so
\begin{equation}\label{Delta_2,2}
\Delta_{2,2}\coloneqq\Delta_2\cap V_2=\left\{z\in (\Delta\setminus \Delta_1)\cap V_2\ \Big|\ \frac{(1-2x)(y+2)}{1+xy}<1/\sqrt{5}\right\}=\left\{z\in \Delta \ \Big|\ \frac{1}{\sqrt{5}-y}\le x\le 1/2\right\}.
\end{equation}
Next, notice that $\Delta_3\subset \Delta\setminus (\Delta_1\cup \Delta_2)\subset V_1\cap \G^{-1}(V_1)$.  This implies that for any $z\in \Delta_3$, $p_2=1$ and $q_2=2$, so 
\begin{equation}\label{Delta_3}
\Delta_3=\left\{z\in\Delta\setminus(\Delta_1\cup \Delta_2)\ \Big|\ \frac{(2x-1)(y+2)}{1+xy}<1/\sqrt{5}\right\}=\Delta\backslash (\Delta_1\cup\Delta_2),
\end{equation}
and $\Delta_n=\varnothing$ for $n>3$.  See Figure \ref{hurwitz_fig}.

\begin{figure}[t]
\centering
\includegraphics[width=.45\textwidth]{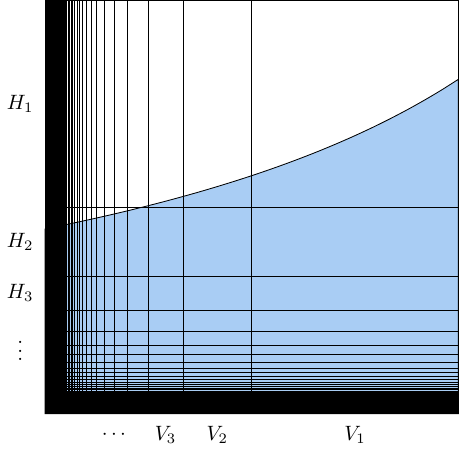}
\includegraphics[width=.45\textwidth]{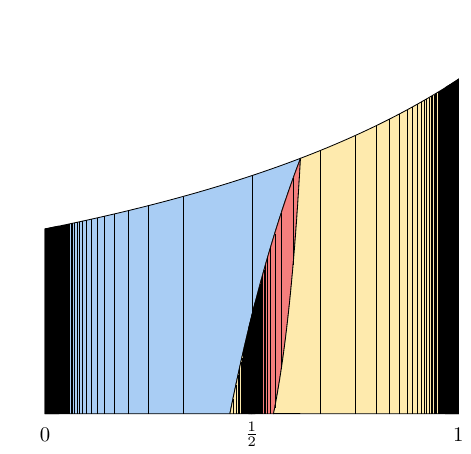}
\caption{Left: The region $\Delta=g^{-1}([0,1/\sqrt{5}))\subset\Omega$.  Right: The regions $\Delta_1(a)$ (blue), $\Delta_{2,1}(a)$ (tan and right of $x=1/2$), $\Delta_{2,2}(a)$ (tan and left of $x=1/2$) and $\Delta_3(a)$ (red).} 
\label{hurwitz_fig}
\end{figure}

\begin{remark}
The fact that the hitting times $j(z)$ are bounded by $3$ implies Borel's theorem (Theorem \ref{thm: Borel}) and hence also (the first statement of) Hurwitz's theorem (Theorem \ref{thm: Hurwitz}). 
\end{remark}

Using \eqref{Delta_1}--\eqref{Delta_3} and the facts that $\Delta_{2,1}\subset V_1$, $\Delta_{2,2}\subset V_2$ and $\Delta_3\subset V_1\cap\G^{-1}(V_1)$, we further refine the sets $\Delta_n$ for each positive integer $a$:
\begin{equation}\label{Delta_1(a)}
\Delta_1(a)\coloneqq \Delta_1\cap V_a= \left\{z\in \Delta_1\ \middle|\ \left\lfloor \frac1x\right\rfloor=a\right\},
\end{equation}
\begin{equation}\label{Delta_2,1(a)}
\Delta_{2,1}(a)\coloneqq \Delta_{2,1}\cap \G^{-1}(V_a)=\left\{z\in \Delta_{2,1}\ \middle|\ \left\lfloor \frac{x}{1-x}\right\rfloor=a\right\},
\end{equation}
\begin{equation}\label{Delta_2,2(a)}
\Delta_{2,2}(a)\coloneqq \Delta_{2,2}\cap \G^{-1}(V_a)=\left\{z\in\Delta_{2,2}\ \middle|\ \left\lfloor\frac{x}{1-2x}\right\rfloor=a\right\}
\end{equation}
and\footnote{We remark that $\Delta_{2,2}(a)=\varnothing$ for $a=1,2,3$.  Indeed, for $z=(x,y)\in\Delta$, $\left\lfloor x/(1-2x)\right\rfloor\le 3$ if and only if $x<4/9$, but the set of points $z$ for which $x<4/9$ belongs to $\Delta_1$.  One can verify that all other sets in \eqref{Delta_1(a)}--\eqref{Delta_3(a)} are non-empty.}
\begin{equation}\label{Delta_3(a)}
\Delta_3(a)\coloneqq \Delta_3\cap \G^{-2}(V_a)=\left\{z\in\Delta_3\ \middle|\ \left\lfloor\frac{1-x}{2x-1}\right\rfloor=a\right\};
\end{equation}
see the right-hand side of Figure \ref{hurwitz_fig}.  We then find
\begin{equation}\label{M_Delta-hurwitz}
M_\Delta(z)=\begin{pmatrix}r_\Delta(z) & p_\Delta(z)\\ s_\Delta(z) & q_\Delta(z)\end{pmatrix}=\begin{cases}
\begin{pmatrix}0 & 1\\ 1 & a\end{pmatrix}, & z\in \Delta_1(a),\\
\begin{pmatrix}1 & a\\ 1 & a+1\end{pmatrix}, & z\in \Delta_{2,1}(a),\\
\begin{pmatrix}1 & a\\ 2 & 2a+1\end{pmatrix}, & z\in \Delta_{2,2}(a),\\
\begin{pmatrix}1 & a+1\\ 2 & 2a+1\end{pmatrix}, & z\in \Delta_3(a).
\end{cases}
\end{equation}
Using this, we may compute the expansion $x=[\beta_0^\Delta;\alpha_0^\Delta/\beta_1^\Delta,\alpha_1^\Delta/\beta_2^\Delta,\dots]$ of any irrational $x\in (0,1)$, as demonstrated in the following example.

\begin{eg}\label{eg: hurwitz}
Let $z=(x,0)$ with $x=\pi-3$.  Using \eqref{G_Delta}, \eqref{Delta_1(a)}--\eqref{Delta_3(a)} and \eqref{M_Delta-hurwitz}, one computes the subsequent regions entered by the $\G_\Delta$-orbit $z_k^\Delta=\G_\Delta^k(z)$ of $z$:
\begin{align*}
z_0^\Delta&\in \Delta_1(7), & z_1^\Delta&\in \Delta_1(15), & z_2^\Delta&\in \Delta_{2,1}(292), & z_3^\Delta&\in \Delta_{2,1}(1), & z_4^\Delta&\in \Delta_{2,1}(2), & z_5^\Delta&\in \Delta_{2,1}(3),\\
z_6^\Delta&\in \Delta_{2,1}(14), & z_7^\Delta&\in \Delta_1(2), & z_8^\Delta&\in \Delta_3(2), & z_9^\Delta&\in \Delta_1(2), & z_{10}^\Delta&\in \Delta_1(2), & z_{11}^\Delta&\in \Delta_1(2),\dots.
\end{align*}
Using this, Proposition \ref{prop: induced_determines_cf_alg} and \eqref{M_Delta-hurwitz}, we compute
\begin{multline*}
x=[\beta_0^\Delta;\alpha_0^\Delta/\beta_1^\Delta,\alpha_1^\Delta/\beta_2^\Delta,\dots]\\
=[0;1/7,1/16,-1/294,-1/3,-1/4,-1/5,-1/15,1/(5/2),(1/2)/5,2/2,1/2\dots].
\end{multline*}
The convergents,
\begin{multline*}
\left(\frac{P_k^\Delta}{Q_k^\Delta}\right)_{k\ge 0}=\Bigg(\frac01,\frac17,\frac{16}{113},\frac{4703}{33215},\frac{14093}{99532},\frac{51669}{364913},\frac{244252}{1725033},\\
\frac{3612111}{25510582},\frac{18549059}{131002976},\frac{48178703}{340262731},\frac{114906465}{811528438},\frac{277991633}{1963319607},\dots\Bigg),
\end{multline*}
are precisely those rationals $p/q$ satisfying Hurwitz's theorem (Theorem \ref{thm: Hurwitz}).  The corresponding approximation coefficients are plotted in Figure \ref{fig: approx_coefs}.
\end{eg}


\bibliography{bib/bib}
\bibliographystyle{acm}

\end{document}